\documentclass[reqno]{amsart}

\usepackage{latexsym,amssymb,amsthm,amsmath,amscd}

\theoremstyle{plain}
\newtheorem{theorem}{Theorem}[section]

\newtheorem{lemma}{Lemma}[section]

\newtheorem{corollary}{Corollary}[section]

\newtheorem{definition}{Definition}[section]
\newtheorem{example}{Example}[section]

\theoremstyle{remark}
\newtheorem{remark}{Remark}[section]

\numberwithin{equation}{section}

\def\<{\left < }
\def\>{\right >}
\def\({\left ( }
\def\){\right )}

\begin{document}



\title[Warped Product Pointwise Semi-slant Submanifolds]{Warped Product Pointwise Semi-slant Submanifolds of Almost Contact Manifolds}

\author[I. Mihai]{Ion Mihai$^1$}
\author[S. Uddin]{Siraj Uddin$^2$}
\address{$^1$Faculty of Mathematics, University of Bucharest, Str. Academiei 14, 010014 Bucharest, Romania}
\email{imihai@fmi.unibuc.ro}
\address{$^2$Department of Mathematics, Faculty of Science, King Abdulaziz University, 21589 Jeddah, Saudi Arabia}
\email{siraj.ch@gmail.com}
\author[A. Mihai]{Adela Mihai$^3$}
\address{$^3$Department of Mathematics and Computer Science,Technical University of Civil Engineering Bucharest, Lacul Tei Bvd. 122-124,  020396 Bucharest, Romania}
\email{adela.mihai@utcb.ro}

\keywords{Warped products, pointwise slant submanifolds, pointwise semi-slant submanifolds, Sasakian manifolds, cosymplectic.
\indent {\it Received}: March 13, 2019.\\
\indent {\it Accepted}:} 

\begin{abstract} 
Recently, B.-Y. Chen and O. J. Garay studied pointwise slant submanifolds of almost Hermitian manifolds. By using the notion of pointwise slant submanifolds, we investigate the geometry of pointwise semi-slant submanifolds and their warped products in Sasakian and cosymplectic manifolds. We prove that there  exist no proper pointwise semi-slant warped product submanifold other than contact CR-warped products in Sasakian manifolds. We give non-trivial examples of such submanifolds in cosypmlectic manifolds and obtain several fundamental results, including a characterization for warped product pointwise semi-slant submanifolds.
\end{abstract}

\maketitle

\sloppy
\section{Introduction}
\label{intro}
In \cite{C1}, B.-Y. Chen introduced the notion of slant submanifolds of almost Hermitian manifolds as a natural generalization of holomorphic (invariant) and totally real (anti-invariant) submanifolds. Afterwards, the geometry of slant submanifolds became an active topic of research in differential geometry. Later, A. Lotta \cite{Lotta} has extended this study for almost contact metric manifolds. J. L. Cabrerizo et al. investigated slant submanifolds of a Sasakian manifold \cite{Cab2}. N. Papaghiuc introduced in \cite{Pa}  a class of submanifolds, called semi-slant submanifolds of almost Hermitian manifolds, which are the generalizations of slant and CR-submanifolds. Later on,  Cabrerizo et al. \cite{Cab} extended this idea for semi-slant submanifolds of contact metric manifolds and provided many examples of such submanifolds.

Next, as an extension of slant submanifolds of an almost Hermitian manifold, F. Etayo \cite{Etayo} introduced the notion of pointwise slant submanifolds of almost Hermitian manifolds. Recently, B.-Y. Chen and O. J. Garay \cite{C6} studied pointwise slant submanifolds of almost Hermitian manifolds. They have obtained several fundamental results, in particular, a characterization of these submanifolds. K. S. Park \cite{Park} has extended this study for almost contact metric manifolds. In his definition of pointwise slant submanifolds of almost contact metric manifolds he did not mention whether the structure vector field $\xi$ is either tangent or normal to the submanifold. Recently, B. Sahin studied pointwise semi-slant submanifolds and warped product pointwise semi-slant submanifolds by using the notion of pointwise slant submanifolds \cite{Sahin3}. In \cite{U5}, we modified the definition of pointwise slant submanifolds of an almost contact metric manifold such that the structure vector field $\xi$ is tangent to the submanifold. We have obtained a simple characterization for such submanifolds and studied warped product pointwise pseudo-slant submanifolds of Sasakian manifolds.

In  1969, R. L. Bishop and B. O'Neill  \cite{Bi} introduced and studied warped product manifolds. 30 years later, around the beginning of this century,  B.-Y. Chen initiated in \cite{C3,C4} the study of warped product CR-submanifolds of Kaehler manifolds. Chen's work in this line of research motivated many geometers to study the geometry of warped product submanifolds by using his idea for different structures on manifolds (see, for instance, \cite{Al1}, \cite{Has}, \cite{Mun} and \cite{U1}). For a detailed survey on warped product submanifolds we refer to Chen's books \cite{book,book17} and his survey article \cite{C5} as well.

In \cite{Sahin1}, B. Sahin showed that there exists no proper warped product semi-slant submanifold of Kaehler manifolds. Then, he introduced the notion of warped product hemi-slant submanifolds of Kaehler manifolds \cite{Sahin2}. Recently, he defined and studied warped product pointwise semi-slant submanifolds and showed that there exists a non-trivial warped product pointwise semi-slant submanifold of the form $M_T\times_fM_\theta$ in a Kaehler manifold $\tilde M$, where $M_T$ and $M_\theta$ are invariant and proper pointwise slant submanifolds of $\tilde M$, respectively \cite{Sahin3}. For almost contact metric manifolds, we have seen in \cite{Khan} and \cite{Al} that there are no proper warped product semi-slant submanifolds in cosymplectic and Sasakian manifolds. Then, we have considered warped product pseudo-slant submanifolds (warped product  hemi-slant submanifolds \cite{Sahin2}, in the same sense of almost Hermitian manifolds) of cosymplectic \cite{U2} and Sasakian manifolds \cite{U3}.

Recently, K. S. Park \cite{Park} studied warped product pointwise semi-slant submanifolds of almost contact metric manifolds. He proved that there do not exist warped product pointwise semi-slant submanifolds of the form $M_\theta\times_fM_T$ in $\tilde M$, where $\tilde M$ is either a cosymplectic manifold, a Sasakian manifold or a Kenmotsu manifold such that $M_\theta$ and $M_T$ are proper pointwise slant and invariant submanifolds of $\tilde M$, respectively. Then he provided many examples and obtained several results for warped products by reversing these two factors, including sharp estimations for the squared norm of the second fundamental form in terms of the warping functions. Later, we also extended this idea in  \cite{U5} to warped product pointwise pseudo-slant submanifolds of Sasakian manifolds. In this paper, we study warped product pointwise semi-slant submanifolds of the form $M_T\times M_\theta$ of Sasakian and cosymplectic manifolds .  

The present paper is organized as follows: In Section $2$, we give basic definitions and formulas needed for this paper. Section $3$ is devoted to the study of pointwise semi-slant submanifolds of almost contact metric manifolds. In this section, we define pointwise semi-slant submanifolds and in the definition of pointwise semi-slant submanifolds we assume that the structure vector field $\xi$ is always tangent to the submanifold. We give two non-trivial examples of such submanifolds for the justification of our definition and a result which is useful to the next section. In Section $4$, we study warped product pointwise semi-slant submanifolds of Sasakian and cosymplectic manifolds. We prove that there is no proper pointwise semi-slant warped product $M=M_T\times_fM_\theta$ other than contact CR-warped product in Sasakian manifolds, but if we assume the ambient space is cosymplectic then there exists a non-trivial class of such warped products. In this section,  we obtain several new results which are generalizations of warped product semi-slant submanifolds and contact CR-warped product submanifolds. In Section $5$, we provide nontrivial examples of Riemannian product and warped product pointwise semi-slant submanifolds in Euclidean spaces. 
\section{Preliminaries}

An {\it{almost contact structure}} $(\varphi, \xi, \eta)$ on a $(2n+1)$-dimensional manifold $\tilde{M}$ is defined  by a $(1, 1)$ tensor field $\varphi$, a vector field $\xi$, called {\it{characteristic}} or {\it{Reeb vector field}}, and a 1-form $\eta$ satisfying the following conditions
\begin{equation}
\label{2.1}
\varphi^2=-I+\eta\otimes\xi,~~\eta(\xi)=1,~~\eta\circ\xi=0,~~\eta(\xi)=1,
\end{equation}
where $I:T\tilde M\to T\tilde M$ is the identity map \cite{Bl}. There always exists a Riemannian metric $g$ on an almost contact manifold
$\tilde M$ satisfying the following compatibility condition
\begin{equation}
\label{2.2}
g(\varphi X, \varphi Y)=g(X, Y)-\eta(X)\eta(Y),
\end{equation}
for any $X, Y\in\Gamma(T\tilde M)$, the Lie algebra of vector fields on $\tilde{M}$. This metric $g$ is called a {\it{compatible metric}} and the manifold $\tilde M$ together with the structure $(\varphi, \xi, \eta, g)$ is called an {\it{almost contact metric manifold}}. As an immediate consequence of (\ref{2.2}), one has $\eta(X)=g(X,\xi)$ and $g(\varphi X, Y)=-g(X, \varphi Y)$. If $\xi$ is a Killing vector field with respect to $g$, then the contact metric structure is called a $K$-{\it{contact structure}}. An almost contact metric manifold is called {\em{almost cosymplectic}}  if $d\eta=0$ and $d\varphi=0$ according to D. E. Blair in \cite{Bl}. In particular, a normal almost cosymplectic manifold is called {\em{cosymplectic}} and  satisfies
\begin{align}\label{cosy}
\widetilde\nabla\varphi=0,\quad
\widetilde\nabla\xi=0.
\end{align}

 A normal contact metric manifold is said to be a {\it{Sasakian manifold}}. In terms of the covariant derivative of $\varphi$, the Sasakian condition can be expressed by
\begin{equation}
\label{2.3}
(\tilde\nabla_{X}\varphi)Y=g(X, Y)\xi-\eta(Y)X
\end{equation}
for all $X, Y\in\Gamma(T\tilde M)$, where $\tilde\nabla$ is the Levi-Civita connection of $g$. From the formula (\ref{2.3}), it follows that 
\begin{align}
\label{2.4}
\tilde\nabla_{X}\xi=-\varphi X,
\end{align}
for any $X\in\Gamma(T\tilde M)$. 

Let $M$ be a Riemannian manifold isometrically immersed in $\tilde M$ and denote by the same symbol $g$ the Riemannian metric induced on $M$. Let $\Gamma(TM)$ be the Lie algebra of vector fields in $M$ and $\Gamma(T^\perp M)$ the set of all vector fields normal to $M$. Let $\nabla$ be the Levi-Civita connection on $M$, then the Gauss and Weingarten formulas are respectively given by 
\begin{align}
&\tilde\nabla_XY=\nabla_XY+h(X, Y)\label{2.5},\\
&\tilde\nabla_XN=-A_{N}X+\nabla^\perp_XN\label{2.6}
\end{align}
for any $X, Y\in\Gamma(TM)$ and $N\in\Gamma(T^\perp M)$, where $\nabla^\perp$ is the normal connection in the normal bundle $T^\perp M$ and $A_{N}$ is the shape operator of $M$ with respect to the normal vector $N$. Moreover, $h: TM\times TM\to T^\perp M$ is the second fundamental form of $M$ in $\tilde M$. Furthermore, $A_N$ and $h$ are related by \cite{Ya}
\begin{equation}
\label{2.7}
g(h(X, Y), N)=g(A_NX, Y)
\end{equation}
for any $X, Y\in\Gamma(TM)$ and $N\in\Gamma(T^\perp M)$.

For any $X$ tangent to $M$, we write
\begin{equation}
\label{2.8}
\varphi X=PX+FX,
\end{equation}
where $PX$ and $FX$ are the tangential and normal components of $\varphi X$, respectively. Then $P$ is an endomorphism of the tangent bundle $TM$ and $F$ is a normal bundle valued 1-form on $TM$. Similarly, for any vector field $N$ normal to $M$, we put
\begin{equation}
\label{2.9}
\varphi N=tN+fN,
\end{equation}
where $tN$ and $fN$ are the tangential and normal components of $\varphi N$, respectively. 
Moreover, from (\ref{2.2}) and (\ref{2.8}), we have 
\begin{equation}
\label{2.10}
g(PX, Y)=-g(X, PY),
\end{equation}
for any $X, Y\in\Gamma(TM)$.

Throughout this paper, we assume the structure field $\xi$ is tangent to $M$ otherwise $M$ is a C-totally real submanifold \cite{Lotta}. Let $M$ be a Riemannian manifold isometrically immersed in an almost contact metric manifold $(\tilde M, \varphi, \xi, \eta, g)$. A submanifold $M$ of an almost contact metric manifold $\tilde M$ is said to be slant \cite{Cab2}, if for each non-zero vector $X$ tangent to $M$ at $p\in M$ such that $X$ is not proportional to $\xi_p$, the angle $\theta(X)$ between $\varphi X$ and $T_pM$ is constant, i.e., it does not depend on the choice of $p\in M$ and $X\in T_pM-\langle\xi_p\rangle$. 

A slant submanifold is said to be {\it{proper slant}} if neither $\theta=0$ nor $\theta=\frac{\pi}{2}$. We note that on a slant submanifold if $\theta=0,$~then it is an invariant submanifold and if $\theta=\frac{\pi}{2},$ then it is an anti-invariant submanifold. A slant submanifold is said to be {\it proper slant} if it is neither invariant nor anti-invariant.

As a natural extension of slant submanifolds, F. Etayo \cite{Etayo} introduced pointwise slant submanifolds of an almost Hermitian manifold under the name of quasi-slant submanifolds. Later on, B.-Y. Chen and O.J. Garay studied pointwise slant submanifolds of almost Hermitian manifolds and obtained many interesting results \cite{C6}. In a similar way, K.S. Park \cite{Park} defined and studied pointwise slant submanifols of almost contact metric manifolds. His definition of pointwise slant submanifolds of almost contact metric manifold is similar to the pointwise slant submanifolds of almost Hermitian manifolds, therefore we have modified his definition by considering the structure vector field $\xi$ is tangent to the submanifold and studied pointwise slant submanifolds of almost contact metric manifolds in \cite{U5}.

A submanifold $M$ of an almost contact metric manifold $\tilde M$ is said to be {\it{pointwise slant}} if for any nonzero vector $X$ tangent to $M$ at $p\in M$, such that $X$ is not propotional to $\xi_p$, the angle $\theta(X)$ between $\varphi X$ and $T_p^*M=T_pM-\{0\}$ is independent of the choice of nonzero vector $X\in T_p^*M$. In this case, $\theta$ can be regarded as a function on $M$, which is called the {\it{slant function}} of the pointwise slant submanifold.

We note that every slant submanifold is a pointwise slant submanifold, but the converse may not be true. We also note that a pointwise slant submanifold is {\it{invariant}} (respectively, {\it{anti-invariant}}) if for each point $p\in M$, the slant function $\theta=0$ (respectively, $\theta=\frac{\pi}{2}$). A pointwise slant submanifold is slant if and only if the slant function $\theta$ is constant on $M$. Moreover, a pointwise slant submanifold is proper if neither $\theta=0, \frac{\pi}{2}$ nor $\theta$ is constant.

In \cite{U5}, we have obtained the following characterization theorem.

\begin{theorem}\label{T1} {\rm{\cite{U5}}} Let $M$ be a submanifold of an almost contact metric manifold $\tilde M$ such that $\xi\in\Gamma(TM)$. Then, $M$ is pointwise slant if and only if 
\begin{align}
\label{2.11}
P^2=\cos^2\theta\left(-I+\eta\otimes\xi\right),
\end{align}
for some real valued function $\theta$ defined on the tangent bundle $TM$ of $M$.
\end{theorem}

The following relations are immediate consequences of Theorem \ref{T1}.

Let $M$ be a pointwise slant submanifold of an almost contact metric manifold $\tilde M$. Then, we have
\begin{align}
\label{2.12}
g(PX, PY)=\cos^2\theta\,[g(X, Y)-\eta(X)\eta(Y)],
\end{align}
\begin{align}
\label{2.13}
g(FX, FY)=\sin^2\theta\,[g(X, Y)-\eta(X)\eta(Y)],
\end{align}
for any $X, Y\in\Gamma(TM)$.

The next useful relation for a pointwise slant submanifold of an almost contact metric manifold was obtained in \cite{U5}
\begin{equation}
\label{2.14}
tFX=\sin^2\theta\left(-X+\eta(X)\xi\right),~~~~fFX=-FPX,
\end{equation}
for any $X\in\Gamma(TM)$.

\section{Pointwise semi-slant submanifolds}

Recently, B. Sahin \cite{Sahin3} defined and studied pointwise semi-slant submanifolds of Kaehler manifolds. In this section, we define and study pointwise semi-slant submanifolds of almost contact metric manifolds.

\begin{definition}\label{D1} {\rm{A submanifold $M$ of an almost contact metric manifold $\tilde M$ is said to be a pointwise semi-slant submanifold if there exists a pair of orthogonal distributions ${\mathfrak{D}}$ and ${\mathfrak{D}}^\theta$ on $M$ such that 
\begin{enumerate}
\item[(i)] The tangent bundle $TM$ admits the orthogonal direct decomposition $TM={\mathfrak{D}}\oplus{\mathfrak{D}}^\theta\oplus\langle\xi\rangle$.
\item[(ii)] The distribution ${\mathfrak{D}}$ is invariant under $\varphi$, i.e., $\varphi\left({\mathfrak{D}}\right)={\mathfrak{D}}$.
\item[(iii)] The distribution ${\mathfrak{D}}^\theta$ is pointwise slant with slant function $\theta$.
\end{enumerate}
}}
\end{definition}
Note that the normal bundle $T^\perp M$ of a pointwise semi-slant submanifold $M$ is decomposed as 
\begin{align*}
T^\perp M=F{\mathfrak{D}}^\theta\oplus\nu,\,\,\,F{\mathfrak{D}}^\theta\perp\nu,
\end{align*}
where $\nu$ is an invariant normal subbundle of $T^\perp M$ under $\varphi$.

If we denote the dimensions of ${\mathfrak{D}}$ and ${\mathfrak{D}}^\theta$ by $m_1$ and $m_2$, respectively, then:
\begin{enumerate}
\item [(i)] If $m_1=0$, then $M$ is a pointwise slant submanifold.
\item [(ii)] If $m_2=0$, then $M$ is an invariant submanifold.
\item [(iii)] If $m_1=0$ and $\theta=\frac{\pi}{2}$, then $M$ is an anti-invariant submanifold.
\item [(iv)] If $m_1\neq0$ and $\theta=\frac{\pi}{2}$, then $M$ is a contact CR-submanifold.
\item [(v)] If $\theta$ is constant on $M$, then $M$ is a semi-slant submanifold with slant angle $\theta$.

\end{enumerate}
We also note that a pointwise semi-slant submanifold is {\it{proper}} if neither $m_1,m_2=0$ nor $\theta=0, \frac{\pi}{2}$ and $\theta$ should not be a constant.

Now, we provide the following non-trivial example of a pointwise semi-slant submanifolds of an almost contact metric manifold.

\begin{example}\label{E1} {\rm{Let $({\bf{R}}^7, \varphi, \xi, \eta, g)$ be an almost contact metric manifold with cartesian coordinates $(x_1,\, y_1, \,x_2,\,y_2,\,x_3,\, y_3,\, z)$ and the almost contact structure
\begin{align*}
\varphi\left(\frac{\partial}{\partial x_i}\right)=-\frac{\partial}{\partial y_i},\,\,\,\varphi\left(\frac{\partial}{\partial y_j}\right)=\frac{\partial}{\partial x_j},\,\,\varphi\left(\frac{\partial}{\partial z}\right)=0,\,\,\,1\leq i, j\leq 3,
\end{align*}
where $\xi=\frac{\partial}{\partial z},\,\,\eta=dz$ and $g$ is the standard Euclidean metric on ${\bf{R}}^7$. Then $(\varphi, \xi, \eta, g)$ is an almost contact metric structure on  ${\bf{R}}^7$.  Consider a submanifold $M$ of ${\bf{R}}^7$ defined by 
\begin{align*}
\psi(u, v, w, t, z)=(u+v,\,-u+v,\,t\cos w,\,t\sin w,\,w\cos t,\,w\sin t,\,z),
\end{align*}
such that $w, t~(w\neq t)$ are non vanishing real valued functions on $M$. Then the tangent space $TM$ is spanned by the following vector fields
\begin{align*}
&X_1=\frac{\partial}{\partial x_1}-\frac{\partial}{\partial y_1},~~X_2=\frac{\partial}{\partial x_1}+\frac{\partial}{\partial y_1},\\
&X_3=-t\sin w\frac{\partial}{\partial x_2}+t\cos w\frac{\partial}{\partial y_2}+\cos t\frac{\partial}{\partial x_3}+\sin t\frac{\partial}{\partial y_3},\\
&X_4=\cos w\frac{\partial}{\partial x_2}+\sin w\frac{\partial}{\partial y_2}-w\sin t\frac{\partial}{\partial x_3}+w\cos t\frac{\partial}{\partial y_3},\,\,X_5=\frac{\partial}{\partial z}.
\end{align*}
Thus, we observe that ${\mathfrak{D}}={\mbox{Span}}\{X_1,\,X_2\}$ is an invariant distribution and ${\mathfrak{D}}^\theta={\mbox{Span}}\{X_3,\, X_4\}$ is a pointwise slant distribution with pointwise slant function $\theta=\cos^{-1}((t-w)/\sqrt{(t^2+1)(w^2+1)})$. Hence, $M$ is a pointwise semi-slant submanifold of ${\bf{R}}^7$ such that $\xi=\frac{\partial}{\partial z}$ is tangent to $M$.}}
\end{example}

\begin{example}\label{E2} {\rm{Consider a submanifold of ${\bf{R}}^7$ with almost contact structure $\varphi$ given in Example \ref{E1}. If the immersion $\psi:{\bf{R}}^5\to {\bf{R}}^7$ is given by
\begin{align*}
\psi(u_1, u_2, u_3, u_4, t)=(u_1,\,(u_3^2+u_4^2)/2,\,\cos u_4,,\,-u_2,\,(u_3^2-u_4^2)/2,\,\sin u_4,\,t),\,\,\, u_4\neq0;
\end{align*}
then the tangent space $TM$ is spanned by $X_1,\,X_2,\,X_3,\,X_4$ and $X_5$ where
\begin{align*}
&X_1=\frac{\partial}{\partial x_1},\,\,\,X_2=-\frac{\partial}{\partial y_1},\,\,\,X_3=u_3\frac{\partial}{\partial x_2}+u_3\frac{\partial}{\partial y_2},\\
&X_4=u_4\frac{\partial}{\partial x_2}-u_4\frac{\partial}{\partial y_2}-\sin u_4 \frac{\partial}{\partial x_3}+\cos u_4\frac{\partial}{\partial y_3},\,\,X_5=\frac{\partial}{\partial t}.
\end{align*}
Then, $M$ is a pointwise semi-slant submanifold such that ${\mathfrak{D}}={\mbox{Span}}\{X_1,\,X_2\}$ is an invariant distribution and ${\mathfrak{D}}^\theta={\mbox{Span}}\{X_3,\, X_4\}$ is a pointwise slant distribution with pointwise slant function $\theta=\cos^{-1}(\sqrt{2}\,u_4/\sqrt{1+2u_4^2})$.}}
\end{example}

Now, we obtain the following useful results for semi-slant submanifolds of a Sasakian (or cosymplectic) manifold.

\begin{lemma}\label{L1} Let $M$ be a pointwise semi-slant submanifold of a Sasakian (or cosymplectic) manifold $\tilde M$. Then, we have
\begin{enumerate}
\item[(i)] $\sin^2\theta\,g(\nabla_XY, Z)=g(h(X, \varphi Y), FZ)-g(h(X, Y), FPZ),$
\item[(ii)] $\sin^2\theta\,g(\nabla_ZW, X)=g(h(X, Z), FPW)-g(h(\varphi X, Z), FW)$,
\end{enumerate}
for any $X, Y\in\Gamma({\mathfrak{D}}\oplus\langle\xi\rangle)$ and $Z, W\in\Gamma({\mathfrak{D}}^\theta)$.
\end{lemma}
\begin{proof} The first and second parts of the lemma can be proved in a similar way. For any $X, Y\in\Gamma({\mathfrak{D}}\oplus\langle\xi\rangle)$ and $Z\in\Gamma({\mathfrak{D}}^\theta)$ we have
\begin{align*}
g(\nabla_XY, Z)=g(\tilde\nabla_XY, Z)=g(\varphi\tilde\nabla_XY, \varphi Z).
\end{align*}
From the covariant derivative formula of $\varphi$, we derive
\begin{align*}
g(\nabla_XY, Z)=g(\tilde\nabla_X\varphi Y, \varphi Z)-g((\tilde\nabla_X\varphi)Y, \varphi Z).
\end{align*}
Then from (\ref{2.3}), (\ref{2.8}) and the orthogonality of the two distributions, we find
\begin{align*}
g(\nabla_XY, Z)&=g(\tilde\nabla_X\varphi Y, PZ)+g(\tilde\nabla_X\varphi Y, FZ)\\
&=-g(\tilde\nabla_XPZ,\varphi Y)+g(h(X, \varphi Y), FZ)\\
&=g(\varphi \tilde\nabla_XPZ,Y)+g(h(X, \varphi Y), FZ).
\end{align*}
Again, from the covariant derivative formula of $\varphi$, we get
\begin{align*}
g(\nabla_XY, Z)=g(\tilde\nabla_X\varphi PZ,Y)-g((\tilde\nabla_X\varphi)PZ, Y)+g(h(X, \varphi Y), FZ).
\end{align*}
Using (\ref{2.3}), (\ref{2.8}) and the orthogonality of vector fields, we obtain
\begin{align*}
g(\nabla_XY, Z)=g(\tilde\nabla_XP^2Z,Y)+g(\tilde\nabla_XFPZ,Y)+g(h(X, \varphi Y), FZ).
\end{align*}
Then from (\ref{2.11}) and (\ref{2.6}), we have
\begin{align*}
g(\nabla_XY, Z)&=-\cos^2\theta\,g(\tilde\nabla_XZ, Y)+\sin2\theta\,X(\theta)\,g(Y, Z)-g(h(X, Y), FPZ)\\
&+g(h(X, \varphi Y), FZ).
\end{align*}
From the orthogonality of the two distributions the above equation takes the form
\begin{align*}
g(\nabla_XY, Z)=\cos^2\theta\,g(\tilde\nabla_XY, Z)-g(h(X, Y), FPZ)+g(h(X, \varphi Y), FZ).
\end{align*}
Hence, (i) follows from the above relation. In a similar way we can prove (ii).\end{proof}

\section{Warped product pointwise semi-slant submanifolds}
In \cite{Bi}, R.L. Bishop and B. O'Neill introduced the notion of warped product manifolds as follows: Let $M_1$ and $M_2$ be two Riemannian manifolds with Riemannian metrics $g_1$ and $g_2$, respectively, and a positive differentiable function $f$ on $M_1$. Consider the product manifold $M_1\times M_2$ with its projections $\pi_1:M_1\times M_2\rightarrow M_1$ and $\pi_2:M_1\times M_2\rightarrow M_2$. Then their warped product manifold $M= M_1\times _fM_2$ is the Riemannian manifold $M_1\times M_2=(M_1\times M_2, g)$ equipped with the Riemannian metric
\begin{equation*}
g(X, Y)=g_1({\pi_1}_\star X, {\pi_1}_\star Y)+(f\circ\pi_1)^2g_2({\pi_2}_\star X, {\pi_2}_\star Y),
\end{equation*}
for any vector field $X, Y$ tangent to $M$, where $\star$ is the symbol for the tangent maps. A warped product manifold $M=M_1\times _fM_2$ is said to be {\it{trivial}} or simply a {\it{Riemannian product manifold}} if the warping function $f$ is constant. Let $X$ be a vector field tangent to $M_1$ and $Z$ be an another vector field on $M_2$, then from Lemma 7.3 of \cite{Bi}, we have
\begin{equation}
\label{4.1}
\nabla_XZ=\nabla_ZX=X(\ln f)Z,
\end{equation}
where $\nabla$ is the Levi-Civita connection on $M$. If $M=M_1\times_fM_2$ is a warped product manifold then the base manifold $M_1$ is totally geodesic in $M$ and the fiber $M_2$ is totally umbilical in $M$ \cite{Bi}, \cite{C3}. 

By analogy to CR-warped products which are introduced by B.-Y. Chen in \cite{C3}, we define the warped product pointwise semi-slant submanifolds as follows.

\begin{definition}\label{D2} {\rm A warped product of an invariant and a pointwise slant submanifolds, say $M_T$ and $M_\theta$ of an almost contact metric manifold $\tilde M$ is called a {\it warped product pointwise semi-slant submanifold}.}\end{definition}
A warped product pointwise semi-slant submanifold $M=M_T\times_fM_\theta$ is called {\it proper} if $M_\theta$ is a proper pointwise slant submanifold and $M_T$ is an invariant submanifold of $\tilde M$ with a non-constant function $f$ on $M$.

The non-existence of warped product pointwise semi-slant submanifolds of the form $M_\theta\times_fM_T$ of Kaehler and Sasakian manifolds is proved in \cite{Sahin3} and \cite{Park}. On the other hand, there exist a non-trivial warped product pointwise semi-slant submanifolds of the form $M_T\times M_\theta$ of Kaehler manifolds \cite{Sahin3} and contact metric manifolds \cite{Park}.  

In this section, we study the warped product pointwise semi-slant submanifold of the form $M=M_T\times_fM_\theta$. Notice that a warped product pointwise semi-slant submanifold $M=M_T\times_fM_\theta$ is a warped product contact CR-submanifold if the slant function $\theta=\frac{\pi}{2}$. Similarly,  the warped product pointwise semi-slant submanifold $M=M_T\times_fM_\theta$ is a warped product semi-slant submanifold if $\theta$ is constant on $M$, i.e., $M_\theta$ is a proper slant submanifold.

On a warped product pointwise semi-slant submanifold $M=M_T\times_fM_\theta$ of a Sasakian (or cosymplectic) manifold $\tilde M$, if we consider the structure vector field $\xi$ is tangent to $M$, then either $\xi\in\Gamma(TM_T)$ or $\xi\in\Gamma(TM_\theta)$. 

\begin{remark}\label{Rs1}
When $\xi$ is tangent to $M_\theta$, then it is easy to check that warped product is trivial (see \cite{U1} and \cite{Khan}), therefore we always consider $\xi\in\Gamma(TM_T)$. 
\end{remark}

First, we prove the following non-existence result of pointwise semi-slant warped products.
\begin{theorem} There do not exist any proper pointwise semi-slant warped product submanifolds of Sasakian manifolds other than contact CR-warped products.
\end{theorem}
\begin{proof} If $M=M_\theta\times_fM_T$ be a pointwise semi-slant warped product submanifold. Then, in a similar way of Theorem 11 and Theorem 12 of \cite{Park}, we find that $M$ is a Riemannian product of $M_T$ and $M_\theta$.

On the other hand, if $M=M_T\times_fM_\theta$ and $\xi$ is tangent to $M_\theta$, then by Remark \ref{Rs1}, $M$ is again a Riemannian product of $M_T$ and $M_\theta$. Furthermore, if $\xi$ is tangent to $M_T$, then from (\ref{2.3}) and (\ref{2.4}), we have
\begin{align}\label{oc1}
\nabla_Z\xi+h(Z, \xi)=-PZ-FZ
\end{align}
for any $Z\in\Gamma(TM_\theta)$. Then, equating the tangential component of \eqref{oc1} and using \eqref{4.1}, we obtain 
\begin{align}\label{oc2}
\xi(\ln f)Z=-PZ. 
\end{align}
Taking the inner product with $W\in\Gamma(TM_\theta)$, we find 
\begin{align}\label{oc3}
\xi(\ln f)g(Z, W)=-g(PZ, W)
\end{align}
By polarization identity, we get
\begin{align}\label{oc4}
\xi(\ln f)g(Z, W)=g(PZ, W)
\end{align}
From \eqref{oc3} and \eqref{oc4}, we find $\xi(\ln f)=0$. Put this value in \eqref{oc2}, we get $PZ=0$, which means that $\theta=\frac{\pi}{2}$. Hence, $M$ is a contact CR-warped product, which proves the theorem completely.
\end{proof}

What is next? We find that if we replace the ambient manifold Sasakian to cosymplectic, then there exists a non-trivial class of pointwise semi-slant warped products.

\begin{lemma}\label{L2} Let $M=M_T\times_fM_\theta$ be a warped product pointwise semi-slant submanifold of a cosymplectic manifold $\tilde M$ such that $\xi\in\Gamma(TM_T)$, where $M_T$ is an invariant submanifold and $M_\theta$ is a proper pointwise slant submanifold of $\tilde M$. Then, we have
\begin{align}
\label{4.2}
g(h(X, W), FPZ)-g(h(X, PZ), FW)=\sin2\theta\,X(\theta)\,g(Z, W),
\end{align}
for any $X\in\Gamma(TM_T)$ and $Z, W\in\Gamma(TM_\theta)$.
\end{lemma}
\begin{proof} For any $X\in\Gamma(TM_T)$ and $Z, W\in\Gamma(TM_\theta)$, we have
\begin{align}\label{4.3}
g(\tilde\nabla_XZ, W)=X(\ln f)\,g(Z, W).
\end{align}
On the other hand, we can obtain
\begin{align*}
g(\tilde\nabla_XZ, W)=g(\varphi\tilde\nabla_XZ, \varphi W).
\end{align*}
Using the covariant derivative formula of $\varphi$, we get
\begin{align*}
g(\tilde\nabla_XZ, W)=g(\tilde\nabla_X\varphi Z, \varphi W).
\end{align*}
Then, from (\ref{2.5}), (\ref{2.8}), (\ref{4.1}) and the orthogonality of vector fields, we find
\begin{align*}
g(\tilde\nabla_XZ, W)&=g(\tilde\nabla_XPZ, PW)+g(\tilde\nabla_XPZ, FW)+g(\tilde\nabla_XFZ, \varphi W)\\
&=X(\ln f)\,g(PZ, PW)+g(h(X, PZ), FW)-g(\varphi\tilde\nabla_XFZ, W)\\
&=\cos^2\theta\,X(\ln f)\,g(Z, W)+g(h(X, PZ), FW)-g(\tilde\nabla_X\varphi FZ, W).
\end{align*}
From (\ref{2.9}) and (\ref{2.14}), we derive
\begin{align}
\label{4.4}
\notag g(\tilde\nabla_XZ, W)&=\cos^2\theta\,X(\ln f)\,g(Z, W)+g(h(X, PZ), FW)+\sin^2\theta\,g(\tilde\nabla_XZ, W)\\
&+\sin2\theta\,X(\theta)\,g(Z, W)+g(\tilde\nabla_XFPZ, W).
\end{align}
Hence, the result follows from (\ref{4.3}) and (\ref{4.4}) by using (\ref{2.6})-(\ref{2.7}) and (\ref{4.1}).\end{proof}

\begin{lemma}\label{L3} Let $M=M_T\times_fM_\theta$ be a warped product pointwise semi-slant submanifold of a cosymplectic manifold $\tilde M$ such that $\xi\in\Gamma(TM_T)$, where $M_T$ and $M_\theta$ are invariant and pointwise slant submanifolds of $\tilde M$, respectively. Then
\begin{enumerate}
\item[(i)] $\xi(\ln f)=0,$
\item[(ii)] $g(h(X, Y), FZ)=0,$
\item[(iii)]  $g(h(X, Z), FW)=X(\ln f)\,g(PZ, W)-\varphi X(\ln f)\,g(Z, W)$
\end{enumerate}
for any $X, Y\in\Gamma(TM_T)$ and $Z, W\in\Gamma(TM_\theta)$.
\end{lemma}
\begin{proof} From (\ref{cosy}), (\ref{2.5}) and (\ref{2.8}), we have
\begin{align*}
\nabla_Z\xi+h(Z, \xi)=0,
\end{align*}
for any $Z\in\Gamma(TM_\theta)$. Which implies that $\xi(\ln f)=0$ by using (\ref{4.1}). (ii) is proved in \cite{Park} ( see relation (100) in  \cite{Park}). Now, for any $X\in\Gamma(TM_T)$ and $Z, W\in\Gamma(TM_\theta)$, we have
\begin{align*}
g(h(X, Z),FW)=g(\tilde\nabla_ZX, FW)=g(\tilde\nabla_ZX, \varphi W)-g(\tilde\nabla_ZX, PW).
\end{align*}
Using the covariant derivative formula of the Riemannain connection and (\ref{4.1}), we get
\begin{align*}
g(h(X, Z), FW)=g((\tilde\nabla_Z\varphi)X, W)-g(\tilde\nabla_Z\varphi X, W)-X(\ln f)\,g(Z, PW).
\end{align*}
Then from (\ref{cosy}), (\ref{2.5}) and (\ref{4.1}), we derive
\begin{align*}
g(h(X, Z), FW)=-\varphi X(\ln f)\,g(Z, W)-X(\ln f)\,g(Z, PW),
\end{align*}
which is third part of the lemma. Hence, the proof is complete.\end{proof}

Interchanging $X$ by $\varphi X$, for any $X\in\Gamma(TM_T)$ in Lemma \ref{L3} (iii), we obtain the following relation.
\begin{align}
\label{4.5}
g(h(\varphi X, Z), FW)=X(\ln f)\,g(Z, W)-\varphi X(\ln f)\,g(Z, PW),
\end{align}
for any $X\in\Gamma(TM_T)$ and $Z, W\in\Gamma(TM_\theta)$.

Similarly, Interchange $Z$ by $PZ$, for any $Z\in\Gamma(TM_\theta)$ in Lemma \ref{L3} (iii), we obtain
\begin{align}
\label{4.6}
g(h(X, PZ), FW)=\varphi X(\ln f)\,g(Z, PW)-\cos^2\theta\,X(\ln f)\,g(Z, W)
\end{align}
for any $X\in\Gamma(TM_T)$ and $Z, W\in\Gamma(TM_\theta)$.

Similarly, if we interchange $W$ by $PW$, for any $W\in\Gamma(TM_\theta)$ in Lemma \ref{L3} (iii), then we derive
\begin{align}
\label{4.7}
g(h(X, Z), FPW)=\cos^2\theta\,X(\ln f)\,g(Z, W)-\varphi X(\ln f)\,g(Z, PW),
\end{align}
for any $X\in\Gamma(TM_T)$ and $Z, W\in\Gamma(TM_\theta)$.

\begin{lemma}\label{L7} Let $M=M_T\times_fM_\theta$ be a warped product pointwise semi-slant submanifold of a cosymplectic manifold $\tilde M$ such that $\xi\in\Gamma(TM_T)$, where $M_T$ and $M_\theta$ are invariant and proper pointwise slant submanifolds of $\tilde M$, respectively. Then, we have
\begin{align}
\label{4.8}
g(A_{FW}\varphi X, Z)-g(A_{FPW}X, Z)=\sin^2\theta\,X(\ln f)\,g(Z, W),
\end{align}
for any $X\in\Gamma(TM_T)$ and $Z, W\in\Gamma(TM_\theta)$.
\end{lemma}
\begin{proof} Subtracting (\ref{4.7}) from (\ref{4.5}), we get (\ref{4.8}).
\end{proof}

A warped product submanifold $M=M_1\times_fM_2$  is {\it{mixed totally geodesic}} if $h(X, Z)=0$, for any $X\in\Gamma(TM_1)$ and $Z\in\Gamma(TM_2)$.

From Lemma \ref{L7}, we obtain the following result.
\begin{theorem}\label{T2} Let $M=M_T\times_fM_\theta$ be a warped product pointwise semi-slant submanifold of a cosymplectic manifold $\tilde M$. If $M$ is mixed totally geodesic, then either $M$ is warped product of invariant submanifolds or the warping function $f$ is constant on $M$.
\end{theorem}
\begin{proof} From (\ref{4.8}) and the mixed totally geodesic condition, we have
\begin{align*}
\sin^2\theta\,X(\ln f)\,g(Z, W)=0.
\end{align*}
Since $g$ is a Riemannian metric, then either $\sin^2\theta=0$ or $X(\ln f)=0$. Therefore, either $M$ is warped product of invariant submanifolds or $f$ is constant on $M$, thus the proof is complete.\end{proof}
\begin{lemma}\label{L8} Let $M=M_T\times_fM_\theta$ be a warped product pointwise semi-slant submanifold of a cosymplectic manifold $\tilde M$ such that $\xi\in\Gamma(TM_T)$, where $M_T$ and $M_\theta$ are invariant and pointwise slant submanifolds of $\tilde M$, respectively. Then, we have
\begin{align}
\label{4.9}
g(A_{FPZ}W, X)-g(A_{FW}PZ, X)=2\cos^2\theta\,X(\ln f)\,g(Z, W)
\end{align}
for any $X\in\Gamma(TM_T)$ and $Z, W\in\Gamma(TM_\theta)$.
\end{lemma}
\begin{proof} Interchanging $Z$ and $W$ in (\ref{4.7}) and using (\ref{2.10}), we get
\begin{align}
\label{4.10}
g(h(X, W), FPZ)=\cos^2\theta\,X(\ln f)\,g(Z, W)+\varphi X(\ln f)\,g(Z, PW),
\end{align}
for any $X\in\Gamma(TM_T)$ and $Z, W\in\Gamma(TM_\theta)$. Subtracting (\ref{4.6}) from (\ref{4.10}), we find (\ref{4.9}).\end{proof}

Also, with the help of Lemma \ref{L8}, we find the following result.
\begin{theorem}\label{T3} Let $M=M_T\times_fM_\theta$ be a warped product pointwise semi-slant submanifold of a cosymplectic manifold $\tilde M$. If $M$ is mixed totally geodesic, then either $M$ is a contact CR-warped product of the form $M_T\times_fM_\perp$ or the warping function $f$ is constant on $M$.
\end{theorem}
\begin{proof} From (\ref{4.9}) and the mixed totally geodesic condition, we have
\begin{align*}
\cos^2\theta\,X(\ln f)\,g(Z, W)=0.
\end{align*}
Since $g$ is a Riemannian metric, then either $\cos^2\theta=0$ or $X(\ln f)=0$. Therefore, either $M$ is a contact CR-warped product or $f$ is constant on $M$, which ends the proof.
\end{proof}

From Theorem  \ref{T2} and Theorem \ref{T3}, we conclude that:
\begin{corollary}\label{C1} There does not exist any mixed totally geodesic proper warped product pointwise semi-slant submanifold $M=M_T\times_fM_\theta$ of a cosymplectic manifold.
\end{corollary}

Also, from Lemma \ref{L2} and Lemma \ref{L8}, we have the following result.

\begin{theorem}\label{T4} Let $M=M_T\times_fM_\theta$ be a warped product pointwise semi-slant submanifold of a cosymplectic manifold $\tilde M$ such that $\xi\in\Gamma(TM_T)$, where $M_T$ is an invariant submanifold and $M_\theta$ is a pointwise slant submanifold of $\tilde M$. Then, either $M$ is a contact CR-warped product of the form $M=M_T\times_fM_\perp$ or $\nabla(\ln f)=\tan\theta\,\nabla \theta$, for any $X\in\Gamma(TM_T)$, where $\nabla f$ is the gradient of $f$.
\end{theorem}
\begin{proof} From Lemma \ref{L2} and Lemma \ref{L8}, we have
\begin{align*}
\cos^2\theta\{X(\ln f)-\tan\theta\,X(\theta)\}\,g(Z, W)=0.
\end{align*}
Since $g$ is a Riemannian metric, therefore we conclude that either $\cos^2\theta=0$ or $X(\ln f)-\tan\theta\,X(\theta)=0$. Consequently, either $\theta=\frac{\pi}{2}$ or $X(\ln f)=\tan\theta\,X(\theta)$, which proves the theorem completely.\end{proof}

As an application of Theorem \ref{T4}, we have the following consequence.
\begin{remark} {\rm{If we consider that the slant function $\theta$ is constant, i.e., $M_\theta$ is a proper slant submanifold in Theorem  \ref{T4}, then $Z(\ln f)=0$, i.e., there are no warped product semi-slant submanifolds of the form $M_T\times_fM_\theta$ in cosymplectic manifolds. Hence, Theorem 4.1 of \cite{Khan}  is a special case of Theorem \ref{T4}.}}\end{remark} 
In order to give a characterization result for pointwise semi-slant submanifolds of a cosymplectic manifold, we need the following well-known result of Hiepko \cite{Hi}.\\ 

\noindent{\bf{Hiepko's Theorem.}} {\rm{Let ${\mathfrak{D}}_1$ and ${\mathfrak{D}}_2$ be two orthogonal distribution on a Riemannian manifold $M$. Suppose that both ${\mathfrak{D}}_1$ and ${\mathfrak{D}}_2$ are involutive such that ${\mathfrak{D}}_1$ is a totally geodesic foliation and ${\mathfrak{D}}_2$ is a spherical foliation. Then $M$ is locally isometric to a non-trivial warped product $M_1\times_fM_2$, where $M_1$ and $M_2$ are integral manifolds of ${\mathfrak{D}}_1$ and ${\mathfrak{D}}_2$, respectively.}}

\begin{theorem}\label{T5} Let $M$ be a pointwise semi-slant submanifold of a cosymplectic manifold $\tilde M$. Then $M$ is locally a non-trivial warped product submanifold of the form $M_T\times_fM_\theta$, where $M_T$ is an invariant submanifold and $M_\theta$ is a proper pointwise slant submanifold of $\tilde M$ if and only if 
\begin{align}
\label{4.11}
A_{FW}\varphi X-A_{FPW}X=\sin^2\theta\,X(\mu)W,\,\,\,\,\forall\,X\in\Gamma({\mathfrak{D}}\oplus\langle\xi\rangle),\,\,W\in\Gamma({\mathfrak{D}}^\theta),
\end{align}
for some smooth function $\mu$ on $M$ satisfying $Z(\mu)=0$, for any $Z\in\Gamma({\mathfrak{D}}^\theta)$.
\end{theorem}
\begin{proof} Let $M=M_T\times_fM_\theta$ be a warped product pointwise semi-slant submanifold of a cosymplectic manifold $\tilde M$. Then for any $X\in\Gamma(TM_T)$ and $Z, W\in\Gamma(TM_\theta)$, from Lemma \ref{L3} (ii) we have
\begin{align}
\label{4.12}
g(A_{FW}X, Y)=0.
\end{align}
Interchanging $X$ by $\varphi X$ in (\ref{4.12}), we get $g(A_{FW}\varphi X, Y)=0$, which means that $A_{FW}\varphi X$ has no component in $TM_T$. Similarly, if we interchange $W$ by $PW$ in (\ref{4.12}) then, we get $g(A_{FPW}X, Y)=0$, i.e., $A_{FPW}X$ also has no component in $TM_T$. Therefore, $A_{FW}\varphi X-A_{FPW}X$ lies in $TM_\theta$, using this fact with Lemma \ref{L7}, we find (\ref{4.11}).

Conversely, if $M$ is a pointwise semi-slant submanifold such that (\ref{4.11}) holds, then from Lemma \ref{L1} (i), we have
\begin{align*}
g(\nabla_XY, W)=\csc^2\theta\,g(A_{FW}\varphi Y-A_{FPW}Y, X),
\end{align*}
for any $X, Y\in\Gamma({\mathfrak{D}}\oplus\langle\xi\rangle)$ and $W\in\Gamma({\mathfrak{D}}^\theta)$. From (\ref{4.11}), we arrive at
\begin{align*}
g(\nabla_XY, W)=Y(\mu)g(X, W)=0,
\end{align*}
which means that the leaves of the distribution ${\mathfrak{D}}\oplus\langle\xi\rangle$ are totally geodesic in $M$. Also, from Lemma \ref{L1} (ii), we have
\begin{align}
\label{4.13}
g(\nabla_ZW, X)=\csc^2\theta\,g(A_{FPW}X-A_{FW}\varphi X, Z),
\end{align}
for any $Z, W\in\Gamma({\mathfrak{D}}^\theta)$ and $X\in\Gamma({\mathfrak{D}}\oplus\langle\xi\rangle)$. By polarization, we derive
\begin{align}
\label{4.14}
g(\nabla_WZ, X)=\csc^2\theta\,g(A_{FPZ}X-A_{FZ}\varphi X, W).
\end{align}
Subtracting (\ref{4.14}) from (\ref{4.13}), we get
\begin{align*}
\sin^2\theta\,g([Z, W], X)=g(A_{FZ}\varphi X-A_{FPZ}X, W)-g(A_{FW}\varphi X-A_{FPW}X, Z).
\end{align*}
Using (\ref{4.11}), we get
\begin{align*}
\sin^2\theta\,g([Z, W], X)=X(\mu)\,g(Z, W)-X(\mu)\,g(W, Z)=0.
\end{align*}
Since $M$ is proper pointwise semi-slant, then $\sin^2\theta\neq0$, thus we conclude that the pointwise slant distribution ${\mathfrak{D}}^\theta$ is integrable. Let us consider $M_\theta$ be a leaf of ${\mathfrak{D}}^\theta$ and $h^\theta$ is the second fundamental form of $M_\theta$ in $M$. Then from (\ref{4.14}), we have
\begin{align*}
g(h^\theta(Z, W), X)=g(\nabla_ZW, X)=-\csc^2\theta\,g(A_{FW}\varphi X-A_{FPW}X, Z)
\end{align*}
Using (\ref{4.11}), we find that
\begin{align*}
g(h^\theta(Z, W), X)=-X(\mu)\,g(Z, W).
\end{align*}
Then from the definition of the gradient of a function, we arrive at
\begin{align*}
h^\theta(Z, W)=-(\vec\nabla\mu)\,g(Z, W).
\end{align*}
Hence, $M_\theta$ is a totally umbilical submanifold of $M$ with the mean curvature vector $H^\theta=-\vec\nabla\mu$, where $\vec\nabla\mu$ is the gradient of the function $\mu$. Since $Z(\mu)=0$, for any $Z\in\Gamma({\mathfrak{D}}^\theta)$, then we can show that $H^\theta=-\vec\nabla\mu$ is parallel with respect to the normal connection, say $D^n$ of $M_\theta$ in $M$ (see  \cite{Sahin2,Sahin3}, \cite{U2}). Thus, $M_\theta$ is a totally umbilical submanifold of $M$ with a non vanishing parallel mean curvature vector $H^\theta=-\vec\nabla\mu$, i.e., $M_\theta$ is an extrinsic sphere in $M$. Then from Heipko's Theorem \cite{Hi}, we conclude that $M$ is a warped product manifold of $M_T$ and $M_\theta$, where $M_T$ and $M_\theta$ are integral manifolds of ${\mathfrak{D}}\oplus\langle\xi\rangle$ and ${\mathfrak{D}}^\theta$, respectively. Thus, the proof is complete.
\end{proof}
As an application of Theorem \ref{T5}, if we consider $\theta=\frac{\pi}{2}$ in Theorem \ref{T5}, then we have the following result as a special case of Theorem \ref{T5}.
\begin{corollary}\label{C2} {\rm{(Theorem 4.2 of \cite{UK11})}} A proper CR-submanifold $M$ of a cosymplectic manifold $\tilde M$, and tangent to the structure vector field $\xi$ is locally a contact CR-warped product if and only if 
\begin{align}
 A_{\varphi Z}X=-\left(\varphi X(\mu)\right)Z,\,\,\,X\in\Gamma({\mathfrak{D}}\oplus\langle\xi\rangle),\,\,Z\in\Gamma({\mathfrak{D}}^\perp),
 \end{align}
for some function $\mu$ on $M$ satisfying $W\mu= 0$ for all $W\in\Gamma({\mathfrak{D}}^\perp)$.
\end{corollary}
\section{Examples}
In this section, we provide the following non-trivial examples of Riemannian products and pointwise semi-slant warped products in Euclidean spaces.
\begin{example}\label{E3}
\rm{Let $M$ be a submanifold of Euclidean 7-space ${\mathbb{R}}^{7}$ with its cartesian coordinates $(x_1,\,\cdots,\,x_3,\,y_1,\cdots,\,y_3,\, t)$ and the almost contact structure
 \begin{align*}
\varphi\left(\frac{\partial}{\partial x_i}\right)=-\frac{\partial}{\partial y_i},\quad
\varphi\left(\frac{\partial}{\partial y_j}\right)=\frac{\partial}{\partial x_j},
\quad\varphi\left(\frac{\partial}{\partial t}\right)=0,\,\,1\leq i,j\leq 3.
 \end{align*}
If $M$ is given by the equations
 \begin{align*}
x_1=u_1,\,x_2=u_3\cos u_4,\,x_3=u_3^2/2,\,y_1=u_2,\,y_2=u_3\sin u_4,\,y_3=u_4,\,t=t,
\end{align*}
for any non-zero function $u_3$ on $M$, then tangent space $TM$ of $M$ is spanned by $X_1,\,X_2,\,X_3,\,X_4$ and $X_5$, where 
\begin{align*}
&X_1=\frac{\partial}{\partial x_1},\quad X_2=\frac{\partial}{\partial y_1},\quad X_3=\cos u_4\frac{\partial}{\partial x_2}+u_3\frac{\partial}{\partial x_3}+\sin u_4\frac{\partial}{\partial y_2},\\
&X_4=-u_3\sin u_4\frac{\partial}{\partial x_2}+u_3\sin u_4\,\frac{\partial}{\partial y_2}+\frac{\partial}{\partial y_3},\quad X_5=\frac{\partial}{\partial t}.
\end{align*}
Then, $M$ is a pointwise semi-slant submanifold with invariant distribution ${\mathfrak{D}}=\rm{Span}\{X_1, X_2\}$ and the pointwise slant distribution ${\mathfrak{D}^{\theta}}=\rm{Span}\{X_3, X_4\}$. Clearly, the slant function is $\theta=\cos^{-1}(2u_3/\sqrt{1+u_3^2})$. Moreover,  ${\mathfrak{D}}$ and ${\mathfrak{D}}^\theta$ are integrable. If $M_T$ and $M_\theta$ are integral manifolds of ${\mathfrak{D}}$ and ${\mathfrak{D}^{\theta}}$, respectively, then, $M=M_T\times M_\theta$ is a Riemannian product of  $M_T$ and $M_\theta$ in ${\mathbb{R}}^{9}$.}
\end{example}
\begin{example}\label{E4}
\rm{Consider the Euclidean 9-space ${\mathbb{R}}^{9}$ with its cartesian coordinates $(x_1,\,\cdots,\,x_4,\,y_1,\cdots,\,y_4,\, t)$ and the almost contact structure
 \begin{align*}
\varphi\left(\frac{\partial}{\partial x_i}\right)=-\frac{\partial}{\partial y_i},\quad\varphi\left(\frac{\partial}{\partial y_j}\right)=\frac{\partial}{\partial x_j},\quad\varphi\left(\frac{\partial}{\partial t}\right)=0,\,1\leq i,j\leq 4.
 \end{align*}
 Let $M$ be a submanifold of ${\mathbb{R}}^{9}$ defined by the immersion $\psi$ as follows:
 \begin{align*}
\psi(u, v, w, s, t)&=(u+v,\,\frac{1}{2}w^2,\,s\cos w, s\sin w,\,-u+v,\,\frac{1}{2}s^2,\,-w\sin s,\, w\cos s\,, t)
\end{align*}
for any non-zero functions $w$ and $s$. The tangent space of $M$ is spanned by the following vectors
\begin{align*}
&X_1=\frac{\partial}{\partial x_1}-\frac{\partial}{\partial y_1},\,X_2=\frac{\partial}{\partial x_1}+\frac{\partial}{\partial y_1},\\
&X_3=w\frac{\partial}{\partial x_2}-s\sin w\,\frac{\partial}{\partial x_3}+s\cos w\,\frac{\partial}{\partial x_4}-\sin s\,\frac{\partial}{\partial y_3}+\cos v\,\frac{\partial}{\partial y_4},\\
&X_4=\cos w\,\frac{\partial}{\partial x_3}+\sin w\,\frac{\partial}{\partial x_4}+s\frac{\partial}{\partial y_2}-w\cos s\,\frac{\partial}{\partial y_3}-w\sin s\,\frac{\partial}{\partial y_4},\,X_5=\frac{\partial}{\partial t}.
\end{align*}
Then, $M$ is a pointwise semi-slant submanifold such that the structure vector field $\xi=\frac{\partial}{\partial t}$ is tangent to $M$ and ${\mathfrak{D}}=\rm{Span}\{X_1, X_2\}$ is an invariant distribution and ${\mathfrak{D}^{\theta}}=\rm{Span}\{X_3, X_4\}$ is a pointwise slant distribution with slant function $\theta=\cos^{-1}\left(\frac{\left(1-ws\right)\sin(w-s)-ws}{{1+w^2+s^2}}\right)$. It is easy to observe that both the distributions are integrable. If we denote the integral manifolds of ${\mathfrak{D}}$ and ${\mathfrak{D}^{\theta}}$ by $M_T$ and $M_\theta$, respectively then, $M$ is a Riemannian product of invariant and pointwise slant submanifolds in ${\mathbb{R}}^{9}$.}
\end{example}
\begin{example}\label{E5}
\rm{Let $M$ be a submanifold of ${\mathbb{R}}^{13}$ given by the immersion $\psi:{\mathbb{R}}^{5}\to {\mathbb{R}}^{13}$ as follows:
 \begin{align*}
\psi(u_1, v_1, u_2, v_2, t)=&(u_1-v_1,\,u_1\cos(u_2+v_2),\,u_1\sin(u_2+v_2),\,v_2,\,u_1\cos(u_2-v_2),\\
&u_1\sin(u_2-v_2),\,u_1+v_1,\,v_1\cos(u_2+v_2),\,v_1\sin(u_2+v_2),\, u_2,\\
&v_1\cos(u_2-v_2),\,v_1\sin(u_2-v_2),\, t),
\end{align*}
for non-zero functions $u_1$ and $v_1$. We use the almost contact structure from Example \ref{E4}. Then, we have
\begin{align*}
&X_1=\frac{\partial}{\partial x_1}+\cos(u_2+v_2)\frac{\partial}{\partial x_2}+\sin(u_2+v_2)\frac{\partial}{\partial x_3}+\cos(u_2-v_2)\frac{\partial}{\partial x_5}\\
&\,\,\,\,\,\,\,\,\,+\sin(u_2-v_2)\frac{\partial}{\partial x_6}+\frac{\partial}{\partial y_1},\\
&X_2=-\frac{\partial}{\partial x_1}+\frac{\partial}{\partial y_1}+\cos(u_2+v_2)\frac{\partial}{\partial y_2}+\sin(u_2+v_2)\frac{\partial}{\partial y_3}+\cos(u_2-v_2)\frac{\partial}{\partial y_5}\\
&\,\,\,\,\,\,\,\,\,+\sin(u_2-v_2)\frac{\partial}{\partial y_6},\\
&X_3=-u_1\sin(u_2+v_2)\frac{\partial}{\partial x_2}+u_1\cos(u_2+v_2)\,\frac{\partial}{\partial x_3}-u_1\sin(u_2-v_2)\frac{\partial}{\partial x_5}\\
&\,\,\,\,\,\,\,\,\,+u_1\cos(u_2-v_2)\,\frac{\partial}{\partial x_6}-v_1\sin(u_2+v_2)\,\frac{\partial}{\partial y_2},\,+v_1\cos(u_2+v_2)\,\frac{\partial}{\partial y_3}\\
&\,\,\,\,\,\,\,\,\,+\frac{\partial}{\partial y_4}-v_1\sin(u_2-v_2)\frac{\partial}{\partial y_5}+v_1\cos(u_2-v_2)\,\frac{\partial}{\partial y_6},\\
&X_4=-u_1\sin(u_2+v_2)\frac{\partial}{\partial x_2}+u_1\cos(u_2+v_2)\,\frac{\partial}{\partial x_3}+\frac{\partial}{\partial x_4}+u_1\sin(u_2-v_2)\frac{\partial}{\partial x_5}
\end{align*}
\begin{align*}
&\,\,\,\,\,\,\,\,\,-u_1\cos(u_2-v_2)\,\frac{\partial}{\partial x_6}-v_1\sin(u_2+v_2)\,\frac{\partial}{\partial y_2}\,+v_1\cos(u_2+v_2)\,\frac{\partial}{\partial y_3}\\
&\,\,\,\,\,\,\,\,\,+v_1\sin(u_2-v_2)\frac{\partial}{\partial y_5}-v_1\cos(u_2-v_2)\,\frac{\partial}{\partial y_6},\,X_5=\frac{\partial}{\partial t}.
\end{align*}
By easy and direct computations we find that ${\mathfrak{D}}=\rm{Span}\{X_1, X_2\}$ is an invariant distribution and ${\mathfrak{D}^{\theta}}=\rm{Span}\{X_3, X_4\}$ is a pointwise slant distribution with slant function $\theta=\cos^{-1}\left(\frac{1}{1+2u_1^2+2v_1^2}\right)$. Hence, $M$ is a pointwise semi-slant submanifold of ${\mathbb{R}}^{13}$. It is easy to observe that both the distributions are integrable. If we denote the integral manifolds of ${\mathfrak{D}}$ and ${\mathfrak{D}^{\theta}}$ by $M_T$ and $M_\theta$, respectively then, the product metric structure of $M$ is given by
\begin{align*}
g=4(du_1^2+dv_1^2)+(1+2u_1^2+2v_1^2)(du_2^2+dv_2^2)=g_{M_T}+f^2g_{M_\theta}.
\end{align*}
Hence, $M=M_T\times_fM_\theta$ is a warped product submanifold in ${\mathbb{R}}^{13}$ with warping function $f=\sqrt{1+2u_1^2+2v_1^2}$.}
\end{example}


\noindent {\bf Acknowledgement.} The authors thank Professor Bang-Yen Chen and Professor Kwang Soon Park for pointing out an error in an earlier version of this article.



\end{document}